\documentclass[10pt]{article}
\usepackage{graphicx}
\usepackage{latexsym}
\usepackage{amsmath}
\usepackage{amssymb}
\usepackage{amsthm}
\usepackage{hyperref}
\usepackage{multicol}
\usepackage{multirow}
\usepackage{romannum}
\usepackage{tikz}
\usepackage{tikz-cd}
\usepackage{mathtools}
\usepackage{pgf,tikz}

\tikzstyle{vertex}=[circle, draw, inner sep=0pt, minimum size=6pt] 
\newcommand{\vertex}{\node[vertex]}
\usetikzlibrary{decorations.markings}
\def \Q {{\mathbb Q}}
\def \N {{\mathbb N}}
\def \Z {{\mathbb Z}}
\def \R {{\mathbb R}}

\newtheorem{theorem}{Theorem}[section]

\newtheorem{cor}[theorem]{Corollary}

\newtheorem{definition}{Definition}[section]

\title{\textbf{Calkin-Wilf tree}}
\author {K. Siddharth Choudary, A. Satyanarayana Reddy\\
Department of 
Mathematics, Shiv Nadar 
University, India-201314\\ (e-mail: 
sk597@snu.edu.in, satyanarayana.reddy@snu.edu.in).
  }
\date{}
\begin{document}
\maketitle
\section{Introduction and preliminaries}
The Calkin-Wilf tree named after Neil Calkin and Herb Wilf. They used this tree in~\cite{C:W} to  enumerate rational numbers in a novel approach. 
The Calkin-Wilf tree is a rooted    binary tree, where each vertex (or fraction)  has a left and a right child. 
The vertices of this tree are labeled by fractions.  The root node is labeled with $\frac{1}{1}.$ 
If the label of a vertex is  $\frac{a}{b}$, then the labels of  its  left 
 and right children respectively are $\frac{a}{a+b}$ and $\frac{a+b}{b}.$ We denote $\gcd(a,b)$ as $(a,b).$

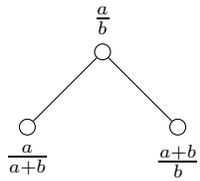
 \begin{tikzpicture}
	\vertex (c1) at (1,2) [label=above:$\frac{a}{b}$]{};
	\vertex (c2) at (0,1) [label=below:$\frac{a}{a+b}$]{};
	\vertex (c3) at (2,1)  [label=below:$\frac{a+b}{b}$]{};
		\path[-]
		(c1) edge (c2)
		(c1) edge (c3)
		
		;
		
\end{tikzpicture}\\

The following rooted tree is  the Calkin Wilf tree (CW-tree) of height $5.$
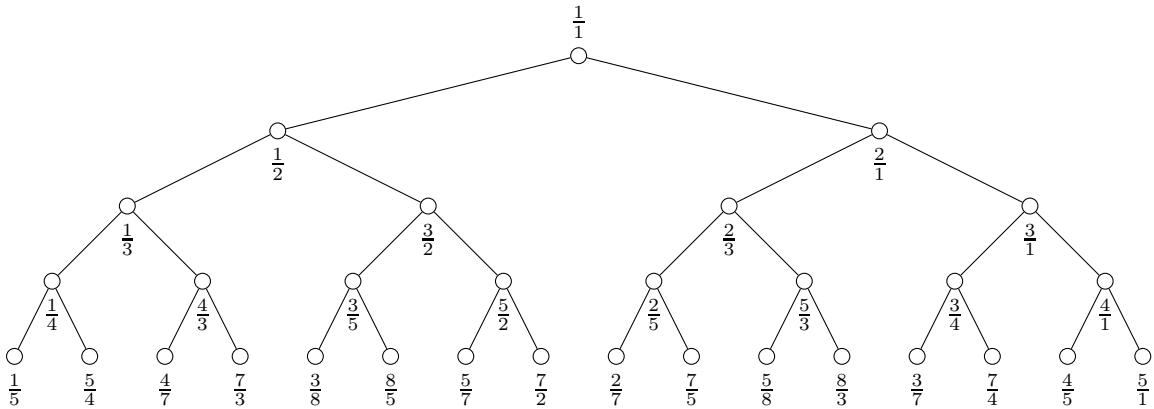
\begin{figure}[h]
\unitlength=0.65mm
 \begin{tikzpicture}
	\vertex (c1) at (7.5,0) [label=above:$\frac{1}{1}$]{};
	\vertex (c2) at (3.5,-1) [label=below:$\frac{1}{2}$]{};
	\vertex (c3) at (11.5,-1)  [label=below:$\frac{2}{1}$]{};
	\vertex (c4) at (1.5,-2) [label=below:$\frac{1}{3}$]{};
	\vertex (c5) at (5.5,-2)  [label=below:$\frac{3}{2}$]{};
	\vertex (c6) at (9.5,-2) [label=below:$\frac{2}{3}$]{};
	\vertex (c7) at (13.5,-2) [label=below:$\frac{3}{1}$]{};	
	\vertex (c8) at (0.5,-3) [label=below:$\frac{1}{4}$]{};
	\vertex (c9) at (2.5,-3)  [label=below:$\frac{4}{3}$]{};
	\vertex (c10) at (4.5,-3) [label=below:$\frac{3}{5}$]{};
	\vertex (c11) at (6.5,-3) [label=below:$\frac{5}{2}$]{};
	\vertex (c12) at (8.5,-3)  [label=below:$\frac{2}{5}$]{};
	\vertex (c13) at (10.5,-3) [label=below:$\frac{5}{3}$]{};
	\vertex (c14) at (12.5,-3) [label=below:$\frac{3}{4}$]{};
	\vertex (c15) at (14.5,-3)  [label=below:$\frac{4}{1}$]{};	
	\vertex (c16) at (0,-4)  [label=below:$\frac{1}{5}$]{};
	\vertex (c17) at (1,-4)  [label=below:$\frac{5}{4}$]{};
	\vertex (c18) at (2,-4)  [label=below:$\frac{4}{7}$]{};
	\vertex (c19) at (3,-4)  [label=below:$\frac{7}{3}$]{};
	\vertex (c20) at (4,-4)  [label=below:$\frac{3}{8}$]{};
	\vertex (c21) at (5,-4)  [label=below:$\frac{8}{5}$]{};
	\vertex (c22) at (6,-4)  [label=below:$\frac{5}{7}$]{};
	\vertex (c23) at (7,-4)  [label=below:$\frac{7}{2}$]{};
	\vertex (c24) at (8,-4)  [label=below:$\frac{2}{7}$]{};
	\vertex (c25) at (9,-4)  [label=below:$\frac{7}{5}$]{};
	\vertex (c26) at (10,-4)  [label=below:$\frac{5}{8}$]{};
	\vertex (c27) at (11,-4)  [label=below:$\frac{8}{3}$]{};
	\vertex (c28) at (12,-4)  [label=below:$\frac{3}{7}$]{};
	\vertex (c29) at (13,-4)  [label=below:$\frac{7}{4}$]{};
	\vertex (c30) at (14,-4)  [label=below:$\frac{4}{5}$]{};
	\vertex (c31) at (15,-4)  [label=below:$\frac{5}{1}$]{};
	
		\path[-]
		(c1) edge (c2)
		(c1) edge (c3)
		(c2) edge (c4)
		(c2) edge (c5)
		(c3) edge (c6)
		(c3) edge (c7)
		(c4) edge (c8)
		(c4) edge (c9)
		(c5) edge (c10)
		(c5) edge (c11)
		(c6) edge (c12)
		(c6) edge (c13)
		(c7) edge (c14)
		(c7) edge (c15)		
		(c8) edge (c16)
		(c8) edge (c17)
		(c9) edge (c18)
		(c9) edge (c19)
		(c10) edge (c20)
		(c10) edge (c21)
		(c11) edge (c22)
		(c11) edge (c23)
		(c12) edge (c24)
		(c12) edge (c25)
		(c13) edge (c26)
		(c13) edge (c27)
		(c14) edge (c28)
		(c14) edge (c29)
		(c15) edge (c30)
		(c15) edge (c31)

		;

\end{tikzpicture}
\caption{CW-tree of height 5}
\label{fig:CW}
\end{figure}

Before stating few immediate properties of CW-tree, we need the following notations. Let $r=\frac{a}{b}$ be a  fraction.
 Then we denote trace, complexity, simplicity of $r$ as $t(r), c(r), s(r)$ respectively. And they are defined as $t(r)=a+b,c(r)=ab,s(r)=\frac{1}{c(r)}.$ Most of  these observations stated and proved in one of the following~\cite{A:Z, B:B:T, C:W, reny}, we are giving the proofs for the sake of completeness.

\begin{theorem}\label{thm:prop}
 \begin{enumerate}
  \item \label{thm:prop:1}  Every fraction in CW-tree is in reduced form.
  \item \label{thm:prop:2} Every positive rational number appears uniquely in CW-tree.
    \item \label{thm:prop:3} At any given level, denominator of any fraction is equal to numerator of its successive fraction (fraction on its right).
  \item \label{thm:prop:4} The $j^{th}$ vertex in any given level is the reciprocal of $j^{th}$ vertex form the end of that level.
  \item \label{thm:prop:5} Every vertex is the product of its children.
  \item \label{thm:prop:6} Product of all the elements in a given level is $1.$
  \item \label{thm:prop:7} Sum of simplicities of all elements in a level is $1.$
  \item \label{thm:prop:8} Product of complexities of all the elements in a level is a perfect square.
  \item \label{thm:prop:9}Sum of traces of all the elements at a level $n$ is $2\cdot 3^{n-1}$
  \item \label{thm:prop:10}Sum of complexities at level $n$ is equal to sum of squares of traces at level $n-1.$
  \item \label{thm:prop:11}Sum of all elements in a level $n$ is $3\cdot 2^{n-2}-\frac{1}{2}.$
 \end{enumerate}
\end{theorem}
\begin{proof}
We prove these results except Part~\ref{thm:prop:2}  by using Mathematical induction on levels. \\
 Proof of Part~\ref{thm:prop:1}. 
 The only fraction at  level 1 is $\frac{1}{1}$ and  $(1,1)=1.$ Assume that all the fractions at  level $k$ are in reduced form. Since every fraction (or vertex) at level $k+1$ is a children of a fraction at level $k.$ Hence it is sufficient to prove $(a+b,b)=1=(a,a+b).$ But this follows immediately from  $(a,b)=1.$ \\
 
 Proof of Part~\ref{thm:prop:2}. Existence: Let $S$ be the set of all positive rationals in the simplified form which do not occur in the CW-tree. Assume $S$ is nonempty.

 Let $D(S)$ be the set of all denominators of elements in $S.$ Since $D(S)$ is the nonempty subset of $\N$, by well ordering principle it has a least element say $b.$  Let $S_b$ be the set of all rationals in $S$ such that whose denominator is $b.$ Again by well ordering principle, the set of numerators of elements of $S_b$ has a least element say $a.$ If $\frac{a}{b}<1$, then its parent $\frac{a}{b-a}$ occurs in CW-tree as $b-a<b.$ Now if $\frac{a}{b-a}$ is a fraction in the tree, then its left child $\frac{a}{b}$ is also a fraction from the tree. Hence we got the required contradiction. Proof is similar when $\frac{a}{b}<1.$

 Uniqueness: Let $T$ be the non-empty set of all positive rationals  which occur more than once in the CW-tree. Let $D(T)$ be the set of all denominators in $T$. Since $D(T)$ is a non-empty subset of $\N,$ from Well-Ordering principle, 
 $D(T)$ has the smallest element say $b.$ Let $D_b$ be the set of elements of $T$ such that whose denominator is $b.$ Let the smallest element in $D_b$ is $\frac{a}{b}.$
 
  If $\frac{a}{b}<$1,  then its parent  $\frac{a}{b-a}$ occurs at least twice in the tree as $\frac{a}{b}$  occurs at
  least  twice. Which is a contradiction as $b$ is the least denominator in $T.$ Similar situation arises if  $\frac{a}{b}>1.$ Hence the result follows.\\
 
Proof of Part~\ref{thm:prop:3}. Clearly true for level $2.$ Assume two consecutive terms at a level $k$ are $\frac{a}{b}$ and $\frac{b}{c}$. Then the right child of $\frac{a}{b}$ and left child of $\frac{b}{c}$  are $\frac{a+b}{b}$ and $\frac{b}{b+c}$ respectively. And they also satisfying the required property. Hence proved. \\

 Proof of Part~\ref{thm:prop:4}. Clearly true for the factions at level $1$ and level $2.$ Let $\frac{a}{b}$ be the $j^{th}$ vertex from the right and $\frac{b}{a}$ be $j^{th}$ vertex from the left at  level $k$. Then at level $k+1$, it is easy to see that  $(2j-1)^{th}$ and $(2j)^{th}$ vertex from the left are $\frac{a}{a+b}$ and $\frac{a+b}{b}$ (children of $\frac{a}{b}$). And $(2j-1)^{th}$ and $(2j)^{th}$ vertex from the right are $\frac{b}{a+b}$ and $\frac{a+b}{a}$ (children of $\frac{b}{a}$). Thus  the result is true for the fractions at  level $k+1.$\\
 
Proof of Part~\ref{thm:prop:5}. Easy to see.\\
 \\Proof of Part~\ref{thm:prop:6}. From Part~\ref{thm:prop:5}, the product of all elements at level $k$ is equal to product of  all elements at level $k-1.$ By continuing in this order, we get that product of all elements in level $k$ is equal to product of all elements in level $1$ which is $1.$ Also follows from Part~\ref{thm:prop:4}.\\
 
Proof of Part~\ref{thm:prop:7}. Sum of simplicities of children of a vertex $\frac{a}{b}$ is
$$s\left(\frac{a}{a+b}\right)+s\left(\frac{a+b}{b}\right) = \frac{1}{a(a+b)}+\frac{1}{(a+b)b}=\frac{1}{ab}=s\left(\frac{a}{b}\right).$$
 
 Hence  sum of simplicities of all fractions at any level  is equal to the sum of simplicities of fraction at level $1$ which is $1.$\\

Proof of Part~\ref{thm:prop:8}. Product of complexities of children of a vertex $\frac{a}{b}$ is
$$c\left(\frac{a}{a+b}\right)c\left(\frac{a+b}{b}\right)=a(a+b)(a+b)b=(a+b)^2c\left(\frac{a}{b}\right).$$
Thus the product of all complexities at a  level $n$ is $d^2$ times the product of all complexities at level $n-1.$ \\

Proof of Part~\ref{thm:prop:9}. Sum of traces of children of $\frac{a}{b}$ is 
$$t\left(\frac{a}{a+b}\right)+t\left(\frac{a+b}{b}\right)=3(a+b)=3 t\left(\frac{a}{b}\right).$$ 
Thus sum of traces of all elements at   level $n$ is $3$ times the sum of traces of all elements at  level $n-1.$ Sum of traces of elements in level $1$ is $2.$ Hence sum of traces of elements in level $n$ is $2\cdot 3^{n-1}$.\\

Proof of Part~\ref{thm:prop:10}. Sum of complexities of children of $\frac{a}{b}$ is 

$$c\left(\frac{a}{a+b}\right)+c\left(\frac{a+b}{b}\right)=a(a+b)+b(a+b)=(a+b)^2=t\left(\frac{a}{b}\right)^2.$$ \\

Proof of Part~\ref{thm:prop:11}. Let $\frac{a}{b}$ and $\frac{b}{a}$ be fractions at  level $n-1.$   Then sum of children of $\frac{a}{b}$ and $\frac{b}{a}$ is 
$$\frac{a}{a+b}+\frac{a+b}{b}+\frac{b}{a+b}+\frac{a+b}{a}=3+\frac{a}{b}+\frac{b}{a}.$$ Hence the sum of all elements at level $n$ is 
$3\cdot 2^{n-3}+ \mbox{ the sum of all elements at level $n-1$}.$ Therefore sum of all elements at level $n$ is $3\cdot 2^{n-3}+3\cdot 2^{n-4}+\dots+3\cdot 2^{-1}+\mbox{sum of all elements in level $1$} = \frac{3}{2}(2^{n-2}+2^{n-3}+\dots +2+1)+1=\frac{3}{2}(2^{n-1}-1)+1 = 3\cdot 2^{n-2}-\frac{1}{2}.$
 \end{proof}

 \section{Continued fractions and CW-tree}
 \begin{definition}
 A fraction of the form
\begin{equation}\label{eq:CF}
 a_0+\cfrac{1}{a_1+\cfrac{1}{a_2+\cfrac{1}{ a_3+\dotsb +
\cfrac{1}{a_{n-1}+\cfrac{1}{a_n } } }}}
\end{equation}
 is called finite continued fraction,
where $a_0,a_1,\ldots,a_n$ are real numbers and also  $a_1,a_2,\ldots,a_n$ are
positive  where as $a_0$ may be negative. The numbers $a_1,a_2,\ldots,a_n$  are
called {\em partial denominators.} Such a fraction called simple if all $a_i\in
\Z.$
\end{definition}
We denote the continued fraction in the Equation~(\ref{eq:CF}) by
$[a_0;a_1,a_2,\ldots,a_n].$

One can prove the following well known result by induction~\cite{burton}. Its converse can be proved by using Euclidean algorithm. 
\begin{theorem}
Every finite simple continued fraction represents a rational number.
\end{theorem}

Let $r\in \Q,r>1$ and $r=[a_0;a_1,a_2,\ldots,a_n].$ Then it is easy to see that $\frac{1}{r}=[0;a_0,a_1,a_2,\ldots,a_n].$ Thus the sum of terms of continued of $r$ is same as that of $\frac{1}{r}.$
Further if $[a_0;a_1,a_2,\ldots,a_n]$ is a continued fraction with $a_n>1$, then $[a_0;a_1,a_2,\ldots,a_n]=[a_0;a_1,a_2,\ldots,a_n-1,1].$
Thus continued fraction of a rational number is not unique.

Following result provides an interesting property of continued fractions of fractions at a given level in CW-tree.
\begin{theorem}
 Sum of the terms in a continued fraction  of any fraction  from the level $n$ of CW-tree  is $n.$ 
\end{theorem}
\begin{proof}
We prove the result by induction on level number. 
The only fraction of at level $1$ is $1$ and whose continued fraction is $[1].$ Assume the result  is for level $k.$ That is if $r=\frac{a}{b}$ is  a continued fraction at a level $k$ and $[a_0;a_1,a_2,\ldots,a_n]$ is its continued fraction, then $a_0+a_1+\dots+a_n=n.$
Since every fraction at level $k+1$ is a children of a fraction from level $k.$ Hence it is sufficient to prove the result for 
$\frac{a}{a+b}$ and $\frac{a+b}{b}.$ Further from Part~\ref{thm:prop:4} of  Theorem~\ref{thm:prop} we can assume $r=\frac{a}{b}>1.$ 
 Now 
 $$\frac{a}{a+b}=\frac{1}{\frac{a+b}{a}}=\frac{1}{1+\frac{b}{a}}=[0;1,a_0,a_1,a_2,\ldots,a_n],$$
 $$\frac{a+b}{b}=1+\frac{a}{b}=[a_0+1;a_1,a_2,\ldots,a_n].$$
 Hence the result follows.
\end{proof}

Let $r$ be a vertex in CW-tree, we denote the unique path from root node $\frac{1}{1}$ to $r$ as $P(r)$ and it is of the form $R^{a_0}L^{a_1}R^{a_2}\cdots L^{a_n}$ or $L^{b_0}R^{b_1}L^{b_2}\cdots R^{b_n}.$  Here $R$ and $L$ indicates right and left directions respectively.

The following result establishes a one to one correspondence between
continued fraction and $P(r)$ for any fraction on $r$ of CW-tree. 

\begin{theorem}
 If $[a_0;a_1,a_2,\ldots,a_n]$ is a continued fraction of fraction $r$ of CW-tree, then $P(r)$ is $R^{a_0}L^{a_1}R^{a_2}\cdots L^{a_n}$ or $L^{a_0}R^{a_1}L^{a_2}\cdots R^{a_n}$ depending on $n$ is even and or respectively.\\ Conversely if $P(r)\in \{R^{a_0}L^{a_1}R^{a_2}\cdots L^{a_n}, L^{a_0}R^{a_1}L^{a_2}\cdots R^{a_n}\}$, then 
 $r=[a_0;a_1,a_2,\ldots,a_n].$
\end{theorem}

\section{Diagonals of CW-tree}
In this section, we study some sequences called {\em left} and {\em right} diagonals of CW-tree. Here diagonals are sequence of fractions in the tree which share the relative positions on consecutive levels. In order to  study these diagonals  we associate two matrices of size $n\times 2^{n-1}$ to a  CW-tree of height $n.$ One of the matrix corresponding  to height $4$ of CW-tree  is denoted by 
$L^4$ and 
$$L^4=\begin{bmatrix}
      \frac{1}{1} &0&0&0& 0 &0&0&0\\
      \frac{1}{2}&\frac{2}{1}&0 &0& 0 &0&0&0\\
      \frac{1}{3}&\frac{3}{2}&\frac{2}{3} &\frac{3}{1}& 0 &0&0&0\\
      \frac{1}{4}&\frac{4}{3}&\frac{3}{5} &\frac{5}{2}&\frac{2}{5} &\frac{5}{3}&\frac{3}{4}&\frac{4}{1}\\
      \end{bmatrix} $$
here the rows are taken from  the first four levels of the  figure~\ref{fig:CW}. Another matrix denoted $R^n$ is obtained from $L^n$ by keeping zero entries as it is and $(R^n)_{ij}=\frac{1}{(L^n)_{ij}}$
if $(L^n)_{ij}\ne 0.$  If we consider entire CW-tree, then corresponding matrices of infinite size are denoted by $L$ and $R$ respectively.

Nonzero entries from each column of $L$ forms a  sequence called {\em left diagonal} of CW-tree. The sequence corresponding to first column denoted $L_1$ is given by $\frac{1}{1}, \frac{1}{2},\frac{1}{3},\ldots$ or simply $(\frac{1}{n}).$ Hence from definition of $R$ the first right diagonal is $R_1=(\frac{n}{1}).$ It is clear that it is sufficient to study left diagonals. First few left diagonals are given below.
\begin{multicols}{4}
\begin{enumerate}
 \item $L_1=(\frac{1}{n})$
 \item $L_2=(\frac{n+1}{n})$
 \item $L_3=(\frac{n+1}{2n+1})$
 \item $L_4=(\frac{2n+1}{n})$
 \item $L_5=(\frac{n+1}{3n+2})$
 \item $L_6=(\frac{3n+21}{2n+1})$
 \item $L_7=(\frac{2n+1}{3n+1})$
 \item $L_8=(\frac{3n+1}{n})$
 \item $L_9=(\frac{n+1}{4n+3})$
 \item $L_{10}=(\frac{4n+3}{3n+2})$
\end{enumerate}
\end{multicols}

Matthew Gagne~\cite{gagne} conjectured $n^{th}$ terms for few  left diagonals. The following result is in that direction.

\begin{theorem}
 Let $L_n=\left(\frac{aj+b}{cj+d}\right)$ be the $n^{th}$ left diagonal in the CW-tree for n$>$1. Then $L_{2n-1}=\left(\frac{aj+b}{(a+c)j+(b+d)}\right)$ and $L_{2n}=\left(\frac{(a+c)j+(b+d)}{cj+d}\right)$.
 \end{theorem}
 \begin{proof}
  Let $t_{k,n}$ be the $n^{th}$ element in $k^{th}$ level. We have that the children of $t_{k,n}$ are $t_{k+1,2n-1}$ (left child) and $t_{k+1,2n}$ (right child). Let $t_{k,n}$ be the first element in $L_n$. Then $t_{k+1,2n-1}$ and $t_{k+1,2n}$ are the first elements of $L_{2n-1}$ and $L_{2n}$. By induction, we get that the $j^{th}$ elements of $L_{2n-1}$ and $L_{2n}$ are left and right children of $j^{th}$ element of $L_n$ respectively. Hence for n$>$1 $j^{th}$ terms of $L_n$, $L_{2n-1}$ and $L_{2n}$ are $\left(\frac{aj+b}{cj+d}\right)$, $\left(\frac{aj+b}{(a+c)j+(b+d)}\right)$ and $\left(\frac{(a+c)j+(b+d)}{cj+d}\right)$ respectively.
  \end{proof}

Above result can be expressed in terms of tree as follows.

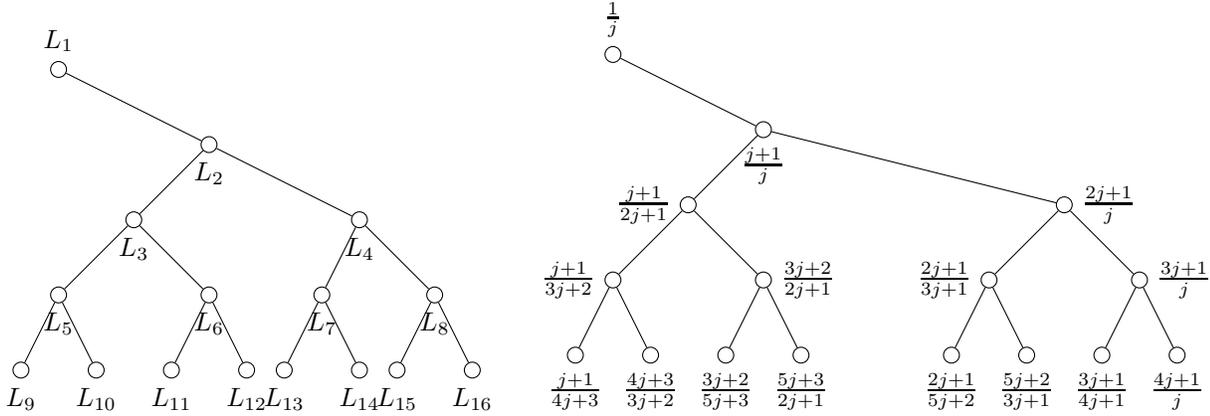
\begin{figure}[h]
\unitlength=0.65mm
\begin{tabular}{ll}
\begin{tikzpicture}
	\vertex (l1) at (3,2) [label=above:$L_1$]{};
	\vertex (l2) at (5,1) [label=below:$L_2$]{};
	\vertex (l3) at (4,0)  [label=below:$L_3$]{};
	\vertex (l4) at (7,0)  [label=below:$L_4$]{};
	\vertex (l5) at (3,-1)  [label=below:$L_5$]{};
	\vertex (l6) at (5,-1)  [label=below:$L_6$]{};
	\vertex (l7) at (6.5,-1)  [label=below:$L_7$]{};
	\vertex (l8) at (8,-1)  [label=below:$L_8$]{};
	\vertex (l9) at (2.5,-2)  [label=below:$L_9$]{};
	\vertex (l10) at (3.5,-2)  [label=below:$L_{10}$]{};
	\vertex (l11) at (4.5,-2)  [label=below:$L_{11}$]{};
	\vertex (l12) at (5.5,-2)  [label=below:$L_{12}$]{};
	\vertex (l13) at (6,-2)  [label=below:$L_{13}$]{};
	\vertex (l14) at (7,-2)  [label=below:$L_{14}$]{};
	\vertex (l15) at (7.5,-2)  [label=below:$L_{15}$]{};
	\vertex (l16) at (8.5,-2)  [label=below:$L_{16}$]{};
		\path[-]
		(l1) edge (l2)
		(l2) edge (l3)
		(l3) edge (l5)
		(l2) edge (l4)
		(l3) edge (l6)
		(l4) edge (l7)
		(l4) edge (l8)
		(l5) edge (l9)
		(l5) edge (l10)
		(l6) edge (l11)
		(l6) edge (l12)
		(l7) edge (l13)
		(l7) edge (l14)
		(l8) edge (l15)
		(l8) edge (l16)
		;
		
\end{tikzpicture}&
\begin{tikzpicture}
	\vertex (l1) at (3,2) [label=above:$\frac{1}{j}$]{};
	\vertex (l2) at (5,1) [label=below:$\frac{j+1}{j}$]{};
	\vertex (l3) at (4,0)  [label=left:$\frac{j+1}{2j+1}$]{};
	\vertex (l4) at (9,0)  [label=right:$\frac{2j+1}{j}$]{};
	\vertex (l5) at (3,-1)  [label=left:$\frac{j+1}{3j+2}$]{};
	\vertex (l6) at (5,-1)  [label=right:$\frac{3j+2}{2j+1}$]{};
	\vertex (l7) at (8,-1)  [label=left:$\frac{2j+1}{3j+1}$]{};
	\vertex (l8) at (10,-1)  [label=right:$\frac{3j+1}{j}$]{};
	\vertex (l9) at (2.5,-2)  [label=below:$\frac{j+1}{4j+3}$]{};
	\vertex (l10) at (3.5,-2)  [label=below:$\frac{4j+3}{3j+2}$]{};
	\vertex (l11) at (4.5,-2)  [label=below:$\frac{3j+2}{5j+3}$]{};
	\vertex (l12) at (5.5,-2)  [label=below:$\frac{5j+3}{2j+1}$]{};
	\vertex (l13) at (7.5,-2)  [label=below:$\frac{2j+1}{5j+2}$]{};
	\vertex (l14) at (8.5,-2)  [label=below:$\frac{5j+2}{3j+1}$]{};
	\vertex (l15) at (9.5,-2)  [label=below:$\frac{3j+1}{4j+1}$]{};
	\vertex (l16) at (10.5,-2)  [label=below:$\frac{4j+1}{j}$]{};
		\path[-]
		(l1) edge (l2)
		(l2) edge (l3)
		(l3) edge (l5)
		(l2) edge (l4)
		(l3) edge (l6)
		(l4) edge (l7)
		(l4) edge (l8)
		(l5) edge (l9)
		(l5) edge (l10)
		(l6) edge (l11)
		(l6) edge (l12)
		(l7) edge (l13)
		(l7) edge (l14)
		(l8) edge (l15)
		(l8) edge (l16)
		;
		
\end{tikzpicture}
\end{tabular}
\caption{Diagonals of CW-tree}
\label{fig:CWD}
\end{figure}

By splitting the coefficient terms and the constant terms into separate trees, we get\\
\textbf{Coefficients:}
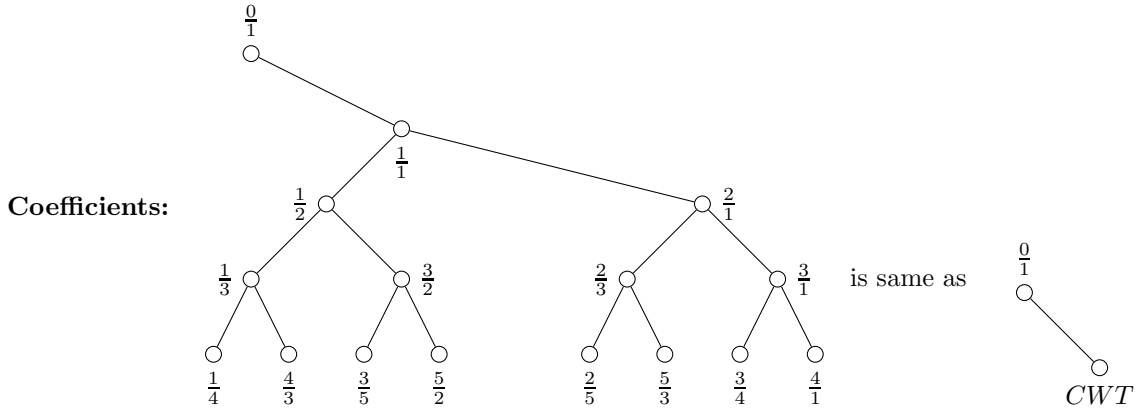
\begin{tabular}{ll}
\begin{tikzpicture}
	\vertex (l1) at (3,2) [label=above:$\frac{0}{1}$]{};
	\vertex (l2) at (5,1) [label=below:$\frac{1}{1}$]{};
	\vertex (l3) at (4,0)  [label=left:$\frac{1}{2}$]{};
	\vertex (l4) at (9,0)  [label=right:$\frac{2}{1}$]{};
	\vertex (l5) at (3,-1)  [label=left:$\frac{1}{3}$]{};
	\vertex (l6) at (5,-1)  [label=right:$\frac{3}{2}$]{};
	\vertex (l7) at (8,-1)  [label=left:$\frac{2}{3}$]{};
	\vertex (l8) at (10,-1)  [label=right:$\frac{3}{1}\hspace{5mm} \mbox{is same as}$]{};
	\vertex (l9) at (2.5,-2)  [label=below:$\frac{1}{4}$]{};
	\vertex (l10) at (3.5,-2)  [label=below:$\frac{4}{3}$]{};
	\vertex (l11) at (4.5,-2)  [label=below:$\frac{3}{5}$]{};
	\vertex (l12) at (5.5,-2)  [label=below:$\frac{5}{2}$]{};
	\vertex (l13) at (7.5,-2)  [label=below:$\frac{2}{5}$]{};
	\vertex (l14) at (8.5,-2)  [label=below:$\frac{5}{3}$]{};
	\vertex (l15) at (9.5,-2)  [label=below:$\frac{3}{4}$]{};
	\vertex (l16) at (10.5,-2)  [label=below:$\frac{4}{1}$]{};
		\path[-]
		(l1) edge (l2)
		(l2) edge (l3)
		(l3) edge (l5)
		(l2) edge (l4)
		(l3) edge (l6)
		(l4) edge (l7)
		(l4) edge (l8)
		(l5) edge (l9)
		(l5) edge (l10)
		(l6) edge (l11)
		(l6) edge (l12)
		(l7) edge (l13)
		(l7) edge (l14)
		(l8) edge (l15)
		(l8) edge (l16)
		;
		
\end{tikzpicture}&
 \begin{tikzpicture}
	\vertex (c1) at (1,2) [label=above:$\frac{0}{1}$]{};
	\vertex (c3) at (2,1)  [label=below:$CWT$]{};
		\path[-]
	
		(c1) edge (c3)

		;
		
\end{tikzpicture}
\end{tabular}\\

Let $b_n$ denotes the sequence of numerators of fractions of CW-tree {\it i.e.,} $(b_n)=(1,1,2,1,3,2,3,1,4,3,5,\ldots).$ Therefore if $t_n=\frac{b_{n-1}}{b_n}$, the the coefficient term of $L_{n+1}$ is $t_n$.\\
\\\textbf{Constants:}
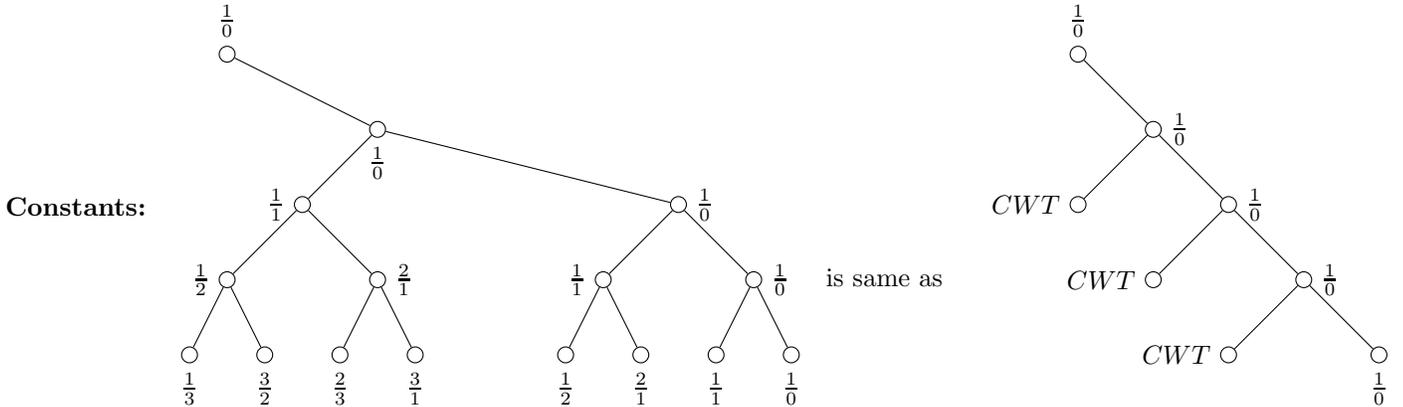
\begin{tabular}{ll}
\begin{tikzpicture}
	\vertex (l1) at (3,2) [label=above:$\frac{1}{0}$]{};
	\vertex (l2) at (5,1) [label=below:$\frac{1}{0}$]{};
	\vertex (l3) at (4,0)  [label=left:$\frac{1}{1}$]{};
	\vertex (l4) at (9,0)  [label=right:$\frac{1}{0}$]{};
	\vertex (l5) at (3,-1)  [label=left:$\frac{1}{2}$]{};
	\vertex (l6) at (5,-1)  [label=right:$\frac{2}{1}$]{};
	\vertex (l7) at (8,-1)  [label=left:$\frac{1}{1}$]{};
	\vertex (l8) at (10,-1)  [label=right:$\frac{1}{0}\hspace{5mm} \mbox{is same as}$]{};
	\vertex (l9) at (2.5,-2)  [label=below:$\frac{1}{3}$]{};
	\vertex (l10) at (3.5,-2)  [label=below:$\frac{3}{2}$]{};
	\vertex (l11) at (4.5,-2)  [label=below:$\frac{2}{3}$]{};
	\vertex (l12) at (5.5,-2)  [label=below:$\frac{3}{1}$]{};
	\vertex (l13) at (7.5,-2)  [label=below:$\frac{1}{2}$]{};
	\vertex (l14) at (8.5,-2)  [label=below:$\frac{2}{1}$]{};
	\vertex (l15) at (9.5,-2)  [label=below:$\frac{1}{1}$]{};
	\vertex (l16) at (10.5,-2)  [label=below:$\frac{1}{0}$]{};
		\path[-]
		(l1) edge (l2)
		(l2) edge (l3)
		(l3) edge (l5)
		(l2) edge (l4)
		(l3) edge (l6)
		(l4) edge (l7)
		(l4) edge (l8)
		(l5) edge (l9)
		(l5) edge (l10)
		(l6) edge (l11)
		(l6) edge (l12)
		(l7) edge (l13)
		(l7) edge (l14)
		(l8) edge (l15)
		(l8) edge (l16)
		;
		
\end{tikzpicture}&
 \begin{tikzpicture}
	\vertex (l1) at (0,2) [label=above:$\frac{1}{0}$]{};
	\vertex (l2) at (1,1) [label=right:$\frac{1}{0}$]{};
	\vertex (l3) at (0,0) [label=left:$CWT$]{};
	\vertex (l4) at (2,0)  [label=right:$\frac{1}{0}$]{};
	\vertex (l7) at (1,-1)  [label=left:$CWT$]{};
	\vertex (l8) at (3,-1)  [label=right:$\frac{1}{0}$]{};
	\vertex (l15) at (2,-2)  [label=left:$CWT$]{};
	\vertex (l16) at (4,-2)  [label=below:$\frac{1}{0}$]{};
		\path[-]
		(l1) edge (l2)
		(l2) edge (l4)
		(l2) edge (l3)
		(l4) edge (l7)
		(l4) edge (l8)
		(l8) edge (l15)
		(l8) edge (l16)
		;
		
\end{tikzpicture}
\end{tabular}\\
For constant term of $L_{n+1}$, let $k=max\{x\in\N\cup\{0\}: n-\sum_{i=o}^{x} 2^{[log_2n]-i}\geq0\}.$ Then, constant term of $L_{n+1}$ is $t_m$ where m = n + $2^{[log_2n]-k-1}$ - $\sum_{i=o}^{k} 2^{[log_2n]-i}$.\\

Therefore $L_{n+1}$=$\frac{a_nj+a_m}{b_nj+b_m}.$

The following results are immediate hence we omit the proofs.
\begin{cor}
 If $L_n$ denote the $n^{th}$ left diagonal. Then $\cup_{i\in \N} L_{2^ni-2^{n-1}}=(n-1,n]\cap \Q^{+}.$
\end{cor}

\begin{cor}
 If $L_n=\frac{aj+b}{cj+d},$ then  $ad-bc=-1.$
\end{cor}

\begin{cor}
 The sequence $L_n$ converges to $t_{n-1}$. Therefore, if $r\ge 0$ is a rational, then there exists unique natural number $n$ such that  $L_n$ converges to $r$.
\end{cor}

\section{The Minkowski question mark function and CW-tree}

The Minkowski question-mark function, denoted $?(x)$  has many strange and unusual properties and is defined by Hermann Minkowski as follows.
\begin{definition}
Let $x\in \R.$ Then 
$$?(x)=\begin{cases}
        a_0 + 2\sum\limits_{n=1}^{\infty} \frac{(-1)^{n+1}}{2^{a_1+a_2+\dots+a_n}} & \mbox{if $x$ is irrational}\\
        a_0 + 2\sum\limits_{n=1}^{m} \frac{(-1)^{n+1}}{2^{a_1+a_2+\dots +a_n}}& \mbox{otherwise},
         \end{cases}$$
where $[a_0;a_1,a_2,\ldots]$ or $[a_0;a_1,a_2,\ldots,a_m]$ is the (or a) continued fraction of $x$ depending on $x$ is irrational and rational respectively. 
\end{definition}
It is easy to see that for every $r\in \R, ?(1+r)=1+?(r)$ and $[y]=[?(y)],$ where $[b]$ denotes greatest integer function of $b\in \R.$ For more properties and to see how one obtains the Minkowski question mark function as the map between the dyadic tree and the Farey tree refer~\cite{LV}. In this section we see how the Minkowski question mark function acts on CW-tree.
\begin{theorem}
 Let $?\left(\frac{a}{b}\right)=x$ and $\left[\frac{a}{b}\right]=n.$ Then 
 $?\left(\begin{tikzcd}
      &\frac{a}{b} \arrow{d} \arrow[swap]{dl} \\
      \frac{a}{a+b}   & \frac{a+b}{b}
   \end{tikzcd}\right)=\begin{tikzcd}
      &x \arrow{d} \arrow[swap]{dl} \\
      1+\frac{x}{2^{n+1}}-\frac{n+2}{2^{n+1}} &1+x
   \end{tikzcd}.$
  \end{theorem}

\begin{proof}
First note that $?\left(\frac{a+b}{b}\right)=?\left(1+\frac{a}{b}\right)=1+?\left(\frac{a}{b}\right)=1+x.$ \\Now we evaluate 
$?\left(\frac{a}{a+b}\right)$ in two cases.
\begin{description}
 \item[$\frac{a}{b} < 1:$] Let $\frac{a}{b} = [0;a_1,a_2,\ldots,a_m].$ Then we have 
 
 \begin{eqnarray*}
  \frac{b}{a} &=& [a_1;a_2,\ldots,a_m],\;\;\;\;
  \frac{a+b}{a} = [a_1+1;a_2,\ldots,a_m],\;\;\;\;
  \frac{a}{a+b}= [0;a_1+1,a_2,\ldots,a_m].
 \end{eqnarray*}
 Consequently, \begin{eqnarray*}
                ?(\frac{a}{a+b}) &=& 2\sum\limits_{i=1}^{m} \frac{(-1)^{i+1}}{2^{(a_1+1)+a_2+\dots+a_i}}
                = 2\sum\limits_{i=1}^{m} \frac{1}{2}\frac{(-1)^{i+1}}{2^{a_1+a_2+\dots+a_i}}
                = \frac{1}{2}?(\frac{a}{b}) = \frac{x}{2}.
               \end{eqnarray*}
Hence $?(\frac{a}{a+b})= 1+\frac{x}{2^{n+1}}-\frac{n+2}{2^{n+1}}$ as $n=0.$
    
  \item[$\frac{a}{b}>1:$]
  
Let $\frac{a}{b} = [a_0;a_1,a_2,\ldots,a_m].$ Then 
$$\frac{b}{a} = [0;a_0,a_1,a_2,\ldots,a_m],\;\; \frac{a+b}{a} = [1;a_0,a_1,a_2,\ldots,a_m]\; and \;
 \frac{a}{a+b} = [0;1,a_0,a_1,a_2,\ldots,a_m].$$
\begin{eqnarray*}
 ?(\frac{a}{a+b})&=& 
2\left(\frac{1}{2} - \frac{1}{2^{1+a_0}} + \sum\limits_{i=1}^{m} \frac{(-1)^{i+1}}{2^{1+a_0+\dots+a_i}}\right)\\
 &=& 1 - \frac{2}{2^{1+a_0}} + 2\left(\frac{1}{2^{1+a_0}}\sum_{i=1}^{m} \frac{(-1)^{i+1}}{2^{a_1+a_2+\dots+a_i}}\right)\\
 &=& 1 - \frac{2}{2^{1+a_0}} + \frac{1}{2^{1+a_0}}\left(-a_0+a_0+2\sum_{i=1}^{m} \frac{(-1)^{i+1}}{2^{a_1+a_2+\dots+a_i}}\right)\\
&=&1 - \frac{2}{2^{1+a_0}} + \frac{-a_0+?(\frac{a}{b})}{2^{1+a_0}}
\end{eqnarray*}
Thus  $?(\frac{a}{a+b}) = 1 + \frac{x-(n+2)}{2^{n+1}}.$ Hence the result.
\end{description}
\end{proof}
Recall that in the CW-tree, every positive rational number uniquely identified by a path. Hence we can assume path $P$ as a positive rational number.  The following result is a direct consequence of above theorem.

\begin{cor} Let $P$ be a path in CW-tree. Then
\begin{multicols}{2}
 \begin{enumerate}
  \item  $?(PR^n) = n+?(P),$
  \item $?(PL^{n+1}) = \frac{1}{2^n}?(PL),$
  \item $?(PLR^nL) = 1-2^{-n}+2^{-(n+1)}?(PL).$
 \end{enumerate}
\end{multicols}
\end{cor}
The following result shows that sum of all the fractions in a level of the CW-tree  is same as the sum of images  of Minkowski question mark function on all the fractions of that level.

\begin{theorem} Let $F_n$ denotes the set of fractions at a level $n$ in the CW-tree. Then 
 $\sum\limits_{r \in F_n}r = \sum\limits_{r \in F_n}?(r).$
\end{theorem}
\begin{proof}
We prove the result by Mathematical induction. First we show that result is true for $n=2.$
$$\sum\limits_{r \in F_2}r = \frac{1}{2}+2 = ?(\frac{1}{2})+?(2) = \sum\limits_{r \in F_2}?(r).$$\\
We know that  $\frac{a}{b}\in F_n\Leftrightarrow\frac{b}{a}\in F_n.$ Hence without loss of generality we can assume that $\frac{a}{b}>1$. Consequently,  if $\frac{a}{b}=[a_0;a_1,a_2,\ldots,a_m],$ then   $\frac{b}{a} = [0;a_0;a_1,a_2,\ldots,a_m].$ Let $?(\frac{a}{b})=x.$
\begin{eqnarray*}
 ?(\frac{b}{a}) &=& 2\sum\limits_{i=0}^{m} \frac{(-1)^{i}}{2^{a_0+a_1+\dots+a_i}}=2\left(\frac{1}{2^{a_0}} - \sum\limits_{i=1}^{m}\frac{(-1)^{i+1}}{2^{a_0+\dots+a_i}}\right)
= \frac{2}{2^{a_0}} - \frac{1}{2^{a_0}}\left(2\sum\limits_{i=1}^{m}\frac{(-1)^{i+1}}{2^{a_1+\dots+a_i}}\right)\\
&=&\frac{2}{2^{a_0}} - \frac{1}{2^{a_0}}\left(-a_0+a_0+2\sum\limits_{i=1}^{m}\frac{(-1)^{i+1}}{2^{a_1+\dots+a_i}}\right)
= \frac{2}{2^{a_0}} - \frac{x-a_0}{2^{a_0}}\\
&=&\frac{2+a_0-x}{2^{a_0}}
\end{eqnarray*}

If $\frac{a}{b}\in F_k$, then $\frac{a}{a+b},\frac{a+b}{b},\frac{b}{a+b},\frac{a+b}{a}\in F_{k+1}.$ And
$$?\left(\frac{a}{a+b}\right)=1+\frac{x-(a+2)}{2^{1+a_0}},\;\;\;\;\; ?\left(\frac{a+b}{b}\right)=1+x,\;\;\;?\left(\frac{b}{a+b}\right)=\frac{2+a_0-x}{2^{1+a_0}},\;\;\;\; ?\left(\frac{a+b}{a}\right)=1+\frac{2+a_0-x}{2^{a_0}}.$$
Thus $$?\left(\frac{a}{a+b}\right)+?\left(\frac{a+b}{b}\right)+?\left(\frac{b}{a+b}\right)+?\left(\frac{a+b}{a}\right) = 3+x+\frac{2+a_0-x}{2^{a_0}} = 3+?\left(\frac{a}{b}\right)+?\left(\frac{b}{a}\right).$$

That is $\sum\limits_{r\in F_n}?(r)=3\cdot 2^{n-3}+\sum\limits_{r\in F_{n-1}}?(r).$ We got the same recurrence relation for sum of fraction at a level $n$ of CW-tree [refer~Part(\ref{thm:prop:11}) of Theorem~\ref{thm:prop}]. Since both sums have the same recurrence relation and  the same starting value, the result follows.
\end{proof}

It is easy see that $$?(L_i)=?(t_{i-1})+\left(\frac{1}{2^{j-1}}\right) 2^{\left([t_{i-1}]-[\log_2(i-1)]\right)},$$ where $L_i$ and $\left(\frac{1}{2^{j-1}}\right)$ are sequences. Hence, we obtain  the following tree, when we apply the Minkowski question function on diagonals of CW-tree.

\begin{figure}[h]
\unitlength=0.65mm
 \begin{tikzpicture}
    \vertex(c0)  at (7.5,1) [label=above:$x$]{};
	\vertex (c1) at (7.5,0) [label=below:$1+x$]{};
	\vertex (c2) at (3.5,-1) [label=above:$\frac{1}{2}+\frac{x}{4}$]{};
	\vertex (c3) at (11.5,-1)  [label=right:$2+x$]{};
	\vertex (c4) at (1.5,-2) [label=left:$\frac{1}{4}+\frac{x}{8}$]{};
	\vertex (c5) at (5.5,-2)  [label=right:$\frac{3}{2}+\frac{x}{4}$]{};
	\vertex (c6) at (9.5,-2) [label=right:$\frac{3}{4}+\frac{x}{8}$]{};
	\vertex (c7) at (13.5,-2) [label=right:$3+x$]{};	
	\vertex (c8) at (0.5,-3) [label=left:$\frac{1}{8}+\frac{x}{16}$]{};
	\vertex (c9) at (2.5,-3)  [label=right:$\frac{5}{4}+\frac{x}{8}$]{};
	\vertex (c10) at (4.5,-3) [label=right:$\frac{5}{8}+\frac{x}{16}$]{};
	\vertex (c11) at (6.5,-3) [label=right:$\frac{5}{2}+\frac{x}{4}$]{};
	\vertex (c12) at (8.5,-3)  [label=right:$\frac{3}{8}+\frac{x}{16}$]{};
	\vertex (c13) at (10.5,-3) [label=right:$\frac{7}{4}+\frac{x}{8}$]{};
	\vertex (c14) at (12.5,-3) [label=right:$\frac{7}{8}+\frac{x}{16}$]{};
	\vertex (c15) at (14.5,-3)  [label=right:$4+x$]{};	
	\vertex (c16) at (0,-4)  [label=left:$\frac{1}{16}+\frac{x}{32}$]{};
	\vertex (c17) at (1,-5)  [label=below:$\frac{9}{8}+\frac{x}{16}$]{};
	\vertex (c18) at (2,-4)  [label=below:$\frac{9}{16}+\frac{x}{32}$]{};
	\vertex (c19) at (3,-5)  [label=below:$\frac{9}{4}+\frac{x}{8}$]{};
	\vertex (c20) at (4,-4)  [label=below:$\frac{5}{16}+\frac{x}{32}$]{};
	\vertex (c21) at (5,-5)  [label=below:$\frac{13}{8}+\frac{x}{16}$]{};
	\vertex (c22) at (6,-4)  [label=below:$\frac{13}{16}+\frac{x}{32}$]{};
	\vertex (c23) at (7,-5)  [label=below:$\frac{7}{2}+\frac{x}{4}$]{};
	\vertex (c24) at (8,-4)  [label=below:$\frac{3}{16}+\frac{x}{32}$]{};
	\vertex (c25) at (9,-5)  [label=below:$\frac{11}{8}+\frac{x}{16}$]{};
	\vertex (c26) at (10,-4)  [label=below:$\frac{11}{16}+\frac{x}{32}$]{};
	\vertex (c27) at (11,-5)  [label=below:$\frac{11}{4}+\frac{x}{8}$]{};
	\vertex (c28) at (12,-4)  [label=below:$\frac{7}{16}+\frac{x}{32}$]{};
	\vertex (c29) at (13,-5)  [label=below:$\frac{15}{8}+\frac{x}{16}$]{};
	\vertex (c30) at (14,-4)  [label=below:$\frac{15}{16}+\frac{x}{32}$]{};
	\vertex (c31) at (15,-4)  [label=right:$5+x$]{};
	
		\path[-]
		(c0) edge (c1)
		(c1) edge (c2)
		(c1) edge (c3)
		(c2) edge (c4)
		(c2) edge (c5)
		(c3) edge (c6)
		(c3) edge (c7)
		(c4) edge (c8)
		(c4) edge (c9)
		(c5) edge (c10)
		(c5) edge (c11)
		(c6) edge (c12)
		(c6) edge (c13)
		(c7) edge (c14)
		(c7) edge (c15)		
		(c8) edge (c16)
		(c8) edge (c17)
		(c9) edge (c18)
		(c9) edge (c19)
		(c10) edge (c20)
		(c10) edge (c21)
		(c11) edge (c22)
		(c11) edge (c23)
		(c12) edge (c24)
		(c12) edge (c25)
		(c13) edge (c26)
		(c13) edge (c27)
		(c14) edge (c28)
		(c14) edge (c29)
		(c15) edge (c30)
		(c15) edge (c31)

		;

\end{tikzpicture}
\caption{?(Figure~\ref{fig:CWD})}
\label{fig:?(CWD)}
\end{figure}

Similarly, one can obtain the coefficients tree  and the constants tree from above tree as we got for Figure~\ref{fig:CWD}.

\end{document}